\documentclass[12pt,oneside,reqno]{amsart}
\usepackage[latin1]{inputenc}
\usepackage{epsfig}
\usepackage{color}
\usepackage{amsthm}
\usepackage{amsmath}
\usepackage{amsfonts}
\usepackage{amssymb}
\usepackage{cancel}
\usepackage{nicefrac}
\usepackage{hyperref}
\usepackage{graphicx}
\usepackage{epsfig}
\usepackage{floatrow}
\usepackage{lmodern}
\usepackage{ae}

\newtheorem{corollary}{Corollary}[section]

\newtheorem{lem}[corollary]{Lemma}
\newtheorem{prp}[corollary]{Proposition}

\newtheorem{thm}[corollary]{Theorem}

\newfont{\sBlackboard}{msbm10 scaled 900}

\newcommand{\mylabel}[1]{\label{#1}
            \ifx\undefined\stillediting
            \else \fbox{$#1$}\fi }
\newcommand{\BE}{\begin{equation}}

\newcommand{\EEQ}{\end{equation}}
\newcommand{\rfb}[1]{\mbox{\rm
   \eqref{#1}}\ifx\undefined\stillediting\else:\fbox{$#1$}\fi}

\newfont{\Blackboard}{msbm10 scaled 1200}

\newfont{\roma}{cmr10 scaled 1200}

\def\n{|\kern -.05cm{|}\kern -.05cm{|}}

\def\N{\rm I\hskip -2pt N} %naturels
\def\C{{\bf \hbox{\sc I\hskip -7pt C}}} %complexe

\newcommand{\mm}    {{\hbox{\hskip 0.5pt}}}

\newcommand{\bluff} {{\hbox{\raise 15pt \hbox{\mm}}}}
scaled\magstep2

\newcommand{\snorm}[1]{\lVert {#1} \rVert}
\makeatletter
\def\section{\@startsection {section}{1}{\z@}{-3.5ex plus -1ex minus
    -.2ex}{2.3ex plus .2ex}{\large\bf}}
\makeatother
\def\be{\begin{equation}}
\def\ee{\end{equation}}

\begin{document}
\thispagestyle{empty}
\title[Degenerate wave equation]{A frequency approach for stabilization of one-dimensional degenerate wave equation}

\author{Akram Ben Aissa}
\address{UR Analysis  and Control  of PDE's, UR 13ES64,
Department of Mathematics, Faculty of Sciences of Monastir, University of Monastir, 5019 Monastir, Tunisia}
\email{akram.benaissa@fsm.rnu.tn}
\author{Mohamed Ferhat}
\address{Departement of Mathematics, Usto  University, Oran 31000,  Algeria}
\email{ferhat22@hotmail.fr}
\author{Ali Segher Kadai}
\address{Laboratory of Mathematics,
Djillali Liabes University,
P. O. Box 89, Sidi Bel Abbes 22000, Algeria}
\email{kadai.alisegher@gmail.com}
\date{}
\begin{abstract}In this paper, we are concerned with the study of stabilization problem for the following strongly degenerate wave equation in one space dimension  
$$w_{tt}(x,t)-\left(x^\alpha w_x(x,t)\right)_x=0$$ where ${\bf\alpha\in [1,2)}$. Thus, using a frequency domain method inspired from  \cite{BT}, we prove the polynomial decays of its total energy with $t^{-\nicefrac{1}{2}}$ decay rate. 
\end{abstract}


\maketitle
 
\tableofcontents

\vfill\break
 
\section{Introduction}
Control problems for degenerate PDE's (and especially for parabolic equations) have received a lot of a attention in the last few years, (see for instance \cite{CMV, CMV1, CMV2}).
So under Carleman estimates with suitable weighted functions, they obtained some observability inequality for the corresponding dual problems.
Therefore, the purpose of this paper is to study stabilization issue for the following degenerate wave equation with $\alpha\in[1,2)$

\begin{equation}\label{s13}
\left\{
\begin{array}{lll}
\begin{split} 
w_{tt}(x,t)-\left(x^\alpha w_x(x,t)\right)_x=0\quad &\text{on}\;(0,1)\times (0,\infty)\\
\left(x^\alpha w_x\right)(0,t)=w_{t}(0,t),\;w(1,t)=0\quad &t\in (0,\infty),\\
w(x,0)=w^0(x),\;w_t(x,0)=w^1(x)\quad &\text{on}\;(0,1).\\
\end{split}
\end{array}
\right.
\end{equation}
So, in a recent paper of Alabau-Cannarsa and Leugering \cite{ACL}, authors studied the same problem as (\ref{s13}) and they proved that exact observability inequality fails for $\alpha\in [1,2)$ via the normal derivative $u_x(1,.)$ and therefore they didn't obtained a such exponentially decay for solutions of (\ref{s13}). More concisely, they studied the following degenerate wave equation
\begin{equation}\label{ver1}
w_{tt}(x,t)-\left(a(x) w_x(x,t)\right)_x=0\quad \text{on}\;(0,1)\times (0,\infty)
\end{equation}
where $a$ is positive function on $]0, 1]$ and vanishes at zero. So under the following linear feedback
\begin{equation}\label{ver2}
w_x(t,1)=-\beta w(t,1)-w_t(t,1),
\end{equation}
they obtained exponential stability of solutions of (\ref{ver1}).

 In this paper, via a frequency domain approach due to Borichev-Tomilov \cite{BT}, we show that system (\ref{s13}) is polynomially stable for $\alpha\in [1,2)$.\\
Here we want to focus on he following ramarks:\\\\
$\bullet$ System (\ref{s13}) under study  is different from one studied on \cite{ACL}. Indeed, the degeneracy is located at $x=0$.\\\\
$\bullet$ The frequency domain method gives us a sharp polynomial decay rate, howver in \cite{ACL}, stabilization is done under the classical energy method due to Komornik \cite{kk}.\\

The outline of this paper as follows.
In section 2, we introduce our notations, functional space and establish the well-posedness of system under study. In section 3, we set our main result concerning stability.
In the last section, we give a numerical simulation of the transfer function for the control system. 
\section{The semigroup setting}
We define the Hilbert space $H^1_{\alpha,r}(0,1)$ as
\begin{equation}
H^1_{\alpha,r}(0,1)=\{u\in L^2(0,1):\;\;x^{\nicefrac{\alpha}{2}}u_x\in L^2(0,1)\;\text{and}\;u(1)=0\},
\end{equation}
equipped with the following inner product
\begin{equation}
\langle f,g\rangle_{H^1_{\alpha,r}(0,1)}=\int_0^1x^{\nicefrac{\alpha}{2}}f_xx^{\nicefrac{\alpha}{2}}\overline{g_x}dx+\int_0^1f\overline{g}dx,
\end{equation}
and its associated norm
\begin{equation}
\snorm{f}^2_{H^1_{\alpha,r}(0,1)}=\snorm{x^{\nicefrac{\alpha}{2}}f_x}^2_{L^2(0,1)}+\snorm{f}^2_{L^2(0,1)}.
\end{equation}
Moreover, we introduce the operator $A_\alpha:D(A)\subset L^2(0,1)\longrightarrow L^2(0,1)$ as
\begin{equation}\label{op1}
\begin{split}
D(A_\alpha)=\{u\in H^1_{\alpha,r}(0,1)&:\;\;(x^\alpha u_x)_x\in L^2(0,1)\},\\
A_\alpha u&=-(x^\alpha u_x)_x,\quad\forall u\in D(A),
\end{split}
\end{equation}

One  can easily check that $A_\alpha$ is self-adjoint positive 
operator  with compact resolvent. Thus, there exists an orthonormal  basis of eigenfunctions denoted by $(\Psi_n)_{n\in\mathbb{N}^*}$ in $L^2(0,1)$ and a real sequence of eigenvalues $(\mu_n)_{n\in\mathbb{N}^*}$
with $\mu_n>0$ and $\mu_n\rightarrow\infty$ such that
\begin{equation}
A_\alpha\Psi_n=\mu_n\Psi_n,\quad\forall n\in\mathbb{N}^*.
\end{equation}
Next, for $s\geq 0$, we introduce the following extrapolated spaces
\begin{equation}
H^s_{\alpha,r}(0,1)=D(A_\alpha^{\nicefrac{s}{2}})=\{u=\sum_{n\geq 1}a_n\Psi_n\;|\;\snorm{u}^2_s=\sum_{n\geq 1}\mu^s_n\left|a_n\right|^2<\infty\}
\end{equation}
and its dual
\begin{equation}
H^{-s}_{\alpha,r}(0,1)=\left(D(A_\alpha^{\nicefrac{s}{2}})\right)^{\prime}.
\end{equation}

Introducing the following Hilbert space
\begin{equation}
\mathcal{H}_\alpha=H_{\alpha,r}^{1}(0,1)\times {L^{2}(0,1)}
\end{equation}
 equipped with the scalar product
\begin{equation}
\langle(u,v)^\mathsf{T},(\tilde{u},\tilde{v})^\mathsf{T}\rangle_{\mathcal{H}_\alpha}= \int_{0}^{1} x^{\nicefrac{\alpha}{2}} u_{x}x^{\nicefrac{\alpha}{2}}\overline{\tilde{u}_x}dx + \int_{0}^{1} v\overline{\tilde{v}}dx
\end{equation}

If we denote by $Z(t)=(w(t),w^\prime(t))^\mathsf{T}$, then the solution of (\ref{s13}) can be written in the abstract Cauchy problem as
\begin{equation}\label{cp1}
\left\{
\begin{array}{ll}
\begin{split}
Z^\prime(t)&=\mathcal{A}_\alpha Z(t)\\
Z(0)&=Z_0,\\
\end{split}
\end{array}
\right.
\end{equation}
where $Z_0=(w^0,w^1)^\mathsf{T}$ and  $\mathcal{A}_\alpha$ is an unbounded operator of $\mathcal{H}_\alpha$ given by
\begin{equation}
\mathcal{A}_\alpha(u,v)^\mathsf{T}=(v,-A_\alpha u)^\mathsf{T},\quad (u,v)\in D(\mathcal{A}_\alpha)
\end{equation}
with
$$D(\mathcal{A}_\alpha)=\{(u,v)\in H_{\alpha,r}^{1}(0,1)\times H_{\alpha,r}^{1}(0,1),\;\; u\in D(A_\alpha)\;\text{and}\;(x^\alpha u_x)(0)=v(0)\}$$

The well-posedness of (\ref{cp1}) is given by the following proposition.
\begin{prp}\label{main2} 
For an initial data $Z_0\in \mathcal{H}_\alpha$, there exists a unique solution $Z\in C([0,\infty),\mathcal{H}_\alpha )$ to system (\ref{cp1}). Moreover, if $Z_0\in D(\mathcal{A}_\alpha)$, then 
$$Z\in C([0,\infty), D(\mathcal{A}_\alpha))\cap C^1([0,\infty), \mathcal{H}_\alpha).$$
Moreover, the energy of system (\ref{s13}) is given by
\begin{equation}
E_w(t)=\frac{1}{2}\int_0^1\left(w_t^2+x^\alpha w_x^2\right)dx,\quad t\geq 0,
\end{equation}
and satisfies
\begin{equation}\label{endec}
E_w(0)-E_w(t)=\int_0^1\left|w_t(0,t)\right|^2dx.
\end{equation}
\end{prp}
\begin{proof}
Using Lumer-Philips theorem \cite{EN}, it suffices to prove that $\mathcal{A}_\alpha$ is maximal-dissipative on $X$. In fact, for all $(u,v)^\mathsf{T}\in D(\mathcal{A}_\alpha)$, we have
\begin{equation*}
\begin{split}
\Re e\langle\mathcal{A}_\alpha(u,v)^\mathsf{T},(u,v)^\mathsf{T}\rangle&=\Re e\langle (v,-A_\alpha u)^\mathsf{T},(u,v)^\mathsf{T}\rangle\\
&=\Re e\left(\int_0^1x^\alpha v_x\overline{u_x}dx+\int_0^1(x^\alpha u_x)_x\overline{v}dx\right)\\
&=\Re e\left(\int_0^1x^\alpha\left(v_x\overline{u_x}-\overline{v_x}u_x\right)dx+\left[(x^\alpha u_x)\overline{v}\right]\bigg\vert_0^1\right)\\
&=-\left|v(0)\right|^2\leq 0,
\end{split}
\end{equation*}

which proves the dissipativeness of $\mathcal{A}_\alpha$.\\\\
Next, let $\lambda>1,\;(f,g)^\mathsf{T}\in X$ and we look for $(u,v)^\mathsf{T}\in D(\mathcal{A}_\alpha)$ such 
\begin{equation}\label{surj}
(\lambda-\mathcal{A}_\alpha)(u,v)^\mathsf{T}=(f,g)^\mathsf{T}.
\end{equation}
That's
\begin{equation}\label{pc1}
\left\{
\begin{array}{ll}
\begin{split}
&\lambda u-v=f\\
&A_\alpha u+\lambda v=g.\\
\end{split}
\end{array}
\right.
\end{equation}
If we suppose that we have found $u$ with an appropriate regularity, then we get
$$v=\lambda u-f\in H^1_{\alpha,r}.$$
Inserting the previous expression in the second equation of (\ref{pc1}) we find that $u$ must satisfy
$$A_\alpha u+\lambda^2 u=g+\lambda f.$$
Multiplying the previous identity by $\overline{w}\in H^1_{\alpha,r}$,  we get
\begin{equation}\label{fs}
    -\int_0^1\left(x^\alpha u_x\right)_x\overline{w}dx+\lambda^2\int_0^1 u\overline{w}dx=\int_0^1\left(g+\lambda f\right)\overline{w}dx.
    \end{equation}
    By a formal integrations by parts, we obtain
		\begin{equation}\label{lm}
    \begin{split}
		\int_0^1x^\alpha u_x\overline{w_x}+\lambda^2\int_0^1 u\overline{w}dx=\int_0^1\left(g+\lambda f\right)\overline{w}dx.
		\end{split}
    \end{equation}
		Thus, equation (\ref{lm}) becomes
		\begin{equation}\label{llm}
		{\bf b}(u,w)={\bf F}(w),\quad w\in V=D(A_\alpha),
		\end{equation}
		where ${\bf b}:H^1_{\alpha,r}\times H^1_{\alpha,r}\longrightarrow\mathbb{R}$ is a bilinear form given by
		$${\bf b}(u,w)=\int_0^1\left(x^\alpha u_x\overline{w_x}+\lambda^2u\overline{w}\right)dx$$
		and ${\bf F}:H^1_{\alpha,r}\longrightarrow\mathbb{R}$ is a linera form given by
		$${\bf F}(w)=\int_0^1\left(g+\lambda
    f\right)\overline{w}dx.$$
		Since ${\bf b}$ is a continuous bilinear coercive form on $H^1_{\alpha,r}$(this follows immediately) and  ${\bf F}$ is a continuous linear form on $H^1_{\alpha,r}$, then by using the Lax-Milgram theorem, we conclude that problem
(\ref{llm}) has a unique solution $u\in H^1_{\alpha,r}$.\\
By an appropriate
choice of $$v=\lambda u-f,$$
we ensured that $(u,v)^\mathsf{T}$ is a
solution of (\ref{surj}), and thus $(\lambda I-\mathcal{A}_\alpha)$ is
surjective. Finally, the Lumer-Phillips theorem  leads to the claim.\\
For the identity (\ref{endec}), it's easy to check.
\end{proof}
\section{Stability results}
First of all, let us recall  the following result due to Borichev and Tomilov\cite{BT} which is will be needed later.
\begin{thm}\label{tmb}(See\cite{BT})
Let $\mathcal{A}$  be the generator of a $C_0$-semigroup of contractions on a Hilbert space
$X$. Then, 
%\begin{enumerate}
%\item we have exponential decay 
%\begin{equation}
%\left\|e^{t\mathcal{A}}U_0\right\|_X\leq Ce^{-tw}\left\|U_0\right\|_{D(\mathcal{A})},\quad t>0,\;w\geq 0,
%\end{equation}
%if and only if
%\begin{equation}
%\displaystyle\lim_{\left|\lambda\right|\rightarrow \infty}\sup\left\|(i\lambda-\mathcal{A})^{-1}\right\|<\infty.
%\end{equation}
%\item we have the polynomial decay
\begin{equation}
\left\|e^{t\mathcal{A}}U_0\right\|_X\leq\frac{C}{t^{\nicefrac{1}{\ell}}}\left\|U_0\right\|_{D(\mathcal{A})},\quad t>0
\end{equation}
for some constant $C>0$, if and only if
\begin{equation}\label{sir1}
i\mathbb{R}\subset\rho(\mathcal{A})
\end{equation}
and
\begin{equation}\label{sir2}
\displaystyle\lim_{\left|\lambda\right|\rightarrow \infty}\sup\frac{1}{\left|\lambda\right|^\ell}\left\|(i\lambda-\mathcal{A})^{-1}\right\|<\infty.
\end{equation}
\end{thm}
Now, we state our main result.
\begin{thm}\label{thp}
Let $\alpha\in[1,2)$. Then, the total energy of system (\ref{cp1}) decays to zero polynomially with the rate $t^{-\nicefrac{1}{2}}$, that's
\begin{equation}
\left\|e^{t\mathcal{A}_\alpha}U_0\right\|_X\leq\frac{C}{t^{\nicefrac{1}{2}}}\left\|U_0\right\|_{D(\mathcal{A}_\alpha)},\quad t>0.
\end{equation}
\end{thm}
\begin{proof}
In view of  Theorem\ref{tmb},  we need firstly to identify the spectrum of $\mathcal{A}_\alpha$ lying on the imaginary axis. We have then to show that :
\begin{enumerate}
\item $\ker\,\left(i\beta-\mathcal{A}_\alpha\right)=\{0\},\quad\forall\beta\in\mathbb{R}$, and
\item $R(\left(i\beta-\mathcal{A}_\alpha\right)=\mathcal{H}_\alpha,\quad\forall\beta\in\mathbb{R}$.
\end{enumerate}
This is the aim  of the two following lemmas.
\begin{lem}
There is no eigenvalue of $\mathcal{A}_\alpha$ on the imaginary axis.
\end{lem}
\begin{proof}
We proceed by contradiction. Assume that there exists at least one $\tilde{\lambda}=i\beta\in\sigma(\mathcal{A}_\alpha),\,\beta\in\mathbb{R}$ on the imaginary axis and \\$\tilde{Z}=(u,v)^\mathsf{T}\in D(\mathcal{A}_\alpha)$ such that 
\begin{equation}\label{dic}
\mathcal{A}_\alpha\tilde{Z}=\tilde{\lambda}\tilde{Z}.
\end{equation}
Then, we have
\begin{eqnarray}
i\beta u-v&=&0\label{e21}\\
A_\alpha u+i\beta v&=&0.\label{e22}
\end{eqnarray}
By taking the inner product of (\ref{dic}) with $\tilde{Z}$ and using the dissipativity of $\mathcal{A}_\alpha$,  we have
\begin{equation}
0=\Re e\langle(i\beta-\mathcal{A}_\alpha)\tilde{Z},\tilde{Z}\rangle=\left|v(0)\right|^2,
\end{equation}
which yields $v(0)=0$. Next, according to (\ref{e21})-(\ref{e22}), we have
\begin{equation}\label{redn}
(x^\alpha u_x)_x+\beta^2u=0,\quad \beta\in\mathbb{R}.
\end{equation}
 The solutions of (\ref{redn}) are given as follows:
\begin{equation}\label{issa}
\beta_n=\kappa j_{\nu,n}\;\;\text{and}\;\;u_n(x)=\frac{\sqrt{2\kappa}}{\left|J^\prime_{\nu}(j_{\nu,n})\right|}x^{\frac{1-\alpha}{2}}J_{\nu}\left(j_{\nu,n}x^\kappa\right),
\end{equation}
where 
\begin{equation}
\nu=\frac{\alpha-1}{2-\alpha},\;\kappa=\frac{2-\alpha}{2}\;\;\text{and}\;J_\nu(x)=\sum_{m\geq 0}\frac{(-1)^m}{m!\Gamma(m+\nu+1)}\left(\frac{x}{2}\right)^{2m+\nu},\;x\geq 0.
\end{equation}
Here $\Gamma(.)$ is the Gamma function, and $\left(j_{\nu,n}\right)_{n\geq 1}$  are the positive zeros of the Bessel function $J_\nu$. See \cite{fa1} for more details.\\
As $u_n\in D(A_\alpha),\;\forall n\geq 1$, then in particular $u_n(1)=0$ which gives us
 $u=0$, and taking account (\ref{e21}) we obtain $v=0$ which contradicts the fact that 
$\tilde{Z}=(u,v)^\mathsf{T}= 0$ is an eigenvector.\\
 The desired result follows.
\end{proof}
\begin{lem}
For all $\beta\in\mathbb{R}$, one has
\begin{equation*}
R(\left(i\beta-\mathcal{A}_\alpha\right)=\mathcal{H}_\alpha.
\end{equation*}
\end{lem}
\begin{proof}
Similarly to the proof of second part of Proposition\ref{main2}, so we omit.
\end{proof}
%Given $(f,g)^\mathsf{T}\in\mathcal{H}_\alpha$, we look for $(u,v)^\mathsf{T}\in D(\mathcal{A}_\alpha)$ such that
%\begin{equation}
%\left(i\beta-\mathcal{A}_\alpha\right)(u,v)=(f,g).
%\end{equation}
%That's
%\begin{equation}\label{main1}
%\left\{
%\begin{array}{ll}
%\begin{split}
%&i\beta u-v=f\\
%&A_\alpha u+i\beta v=g.\\
%\end{split}
%\end{array}
%\right.
%\end{equation}

In order to complete the proof of  Theorem\ref{thp}, it remains to check condition(\ref{sir2}) of Theorem \ref{tmb}. For this end, we proceed by using a contradiction argument. Thus, we assume that (\ref{sir2}) does not hold, then there exist sequences $(\beta_n)_{n},\;\beta_n\in\mathbb{R}_+,\;\beta_n\rightarrow\infty$ and $(U_n)_n$ with $U_n=(u_n,v_n)$ in $D(\mathcal{A}_\alpha),\;n\in\mathbb{N}$, such that
\begin{equation}\label{uni1}
\left\|U_n\right\|_{\mathcal{H}_\alpha}=1,\quad\forall n\in\mathbb{N}
\end{equation}
and
\begin{equation}\label{uni2}
\beta_n^l\left(i\beta_n-\mathcal{A}_\alpha\right)U_n\rightarrow 0,\quad\text{in}\,\mathcal{H}_\alpha\,\text{as}\,n\rightarrow\infty.
\end{equation}
This yields: As $n\rightarrow\infty$
\begin{equation}\label{uni3}
\begin{split}
\beta_n^l\left(i\beta_nu_n-v_n\right)=f_n\rightarrow 0\quad&\text{in}\;H^1_{\alpha,r}(0,1)\\
\beta_n^l\left(A_\alpha u_n+i\beta_nv_n\right)=g_n\rightarrow 0\quad&\text{in}\;L^2(0,1).\\
\end{split}
\end{equation}
Taking into account the following  
\begin{equation}\label{z11}
\beta_n^l\left|v_n(0)\right|^2=\Re e\langle\beta_n^l(i\beta_n-\mathcal{A}_\alpha)U_n,U_n\rangle\leq\left\|\beta_n^l(i\beta_n-\mathcal{A}_\alpha)\right\|,
\end{equation}
we get 
\begin{equation}
\beta_n^l\left|v_n(0)\right|^2,\quad\text{as}\;n\rightarrow\infty.
\end{equation}
On the other hand, we can write
\begin{equation}
\left|v_n^2(1)-v_n^2(0)\right|=\left|\int_0^1v_{n,x}^2dx\right|\leq\left\|v_{n,x}\right\|^2_{L^2}
\end{equation}
which implies by invoking Poincar\'e's inequality
\begin{equation}\label{uni4}
v_n\rightarrow 0\quad\text{in}\;L^2(0,1).
\end{equation}
Multiplying the first equation in (\ref{uni3}) by $i\beta_n$ and summing with the second equation to get
\begin{equation}\label{uni6}
-\beta_n^l\left(x^\alpha u_{n,x}\right)_x-\beta_n^{l+2}u_n=g_n+i\beta_nf_n.
\end{equation}
Now, setting $l=2$ and taking the $L^2$-inner product of (\ref{uni6}) with $(iu_n)$ we arrive after, taking real parts, at
\begin{equation*}
\left\|x^{\nicefrac{\alpha}{2}}u_{n,x}\right\|^2_{L^2}+\beta_n^2\left\|u_n\right\|^2_{L^2}\rightarrow 0.
\end{equation*}
Since $\beta_n\rightarrow\infty$ as $n$ tends to infinity, then it follows that $1\leq \beta_n$ for sufficiently big $n$, so we can write
\begin{equation*}
\left\|x^{\nicefrac{\alpha}{2}}u_{n,x}\right\|^2_{L^2}+\left\|u_n\right\|^2_{L^2}\rightarrow 0,
\end{equation*}
that's
\begin{equation}\label{uni7}
\left\|u_n\right\|_{H^1_{\alpha,r}}\rightarrow 0.
\end{equation}
Combining (\ref{uni4}) and (\ref{uni7}), we have a contradiction with (\ref{uni1}). Thus, (\ref{sir2}) is verified and the proof of Theorem\ref{thp} is complete.
\end{proof}
Now, let us further show the lack of exponential decays for solutions of (\ref{s13}) by using a frequency domain estimate for exponential stability as described in \cite{Hu,Pr}. For this end, we  state  the following result.
\begin{lem}
There exists at least one sequence $(\lambda_n,F_n)$ such that $\lambda_n\rightarrow +\infty$ as $n\rightarrow\infty$ and
\begin{equation}\label{reseq}
\left\|\left(i\lambda_n-\mathcal{A}_\alpha\right)^{-1}F_n\right\|_{\mathcal{H}_\alpha}\rightarrow\infty\quad\text{as}\;\,n\rightarrow\infty,
\end{equation}
with $F_n\in\mathcal{H}_\alpha$ and $\left\|F_n\right\|_{\mathcal{H}_\alpha}$ is bounded, $\tilde{M}$ is positive constant.
\end{lem}
\begin{proof}
Setting the following resolvent equation
\begin{equation}
\left(i\lambda -\mathcal{A}_\alpha\right)U=F,\quad \lambda\in\mathbb{R},
\end{equation}
where $U=(u,v)^\mathsf{T}$ and $F=(f,g)^\mathsf{T}$. That's
\begin{equation}\label{vbg1}
\left\{
 \begin{array}{c}
		\begin{split}
  i\lambda u-v&=f \\
  Au+i\lambda v&=g. \\
	\end{split}
   \end{array}
    \right.
\end{equation}
Choosing $f=0$ and substituting $v=i\lambda u$ into the second equation of (\ref{vbg1}), we get
\begin{equation}\label{vbg2}
\left\{
 \begin{array}{lll}
		\begin{split}
  &\left(x^\alpha u_x\right)_x+\lambda^2 u=g,\quad x\in (0,1), \\
	&\text{with boundary conditions}\\
	&\left(x^\alpha u_x\right)(0)=0,\quad u(1)=0.
	\end{split}
   \end{array}
    \right.
\end{equation}
So, according to \cite{fa1}, the solutions of (\ref{vbg2}) are given as follows:
\begin{equation}\label{issa}
\lambda_n=\kappa j_{\nu,n}\;\;\text{and}\;\;u_n(x)=\frac{\sqrt{2\kappa}}{\left|J^\prime_{\nu}(j_{\nu,n})\right|}x^{\frac{1-\alpha}{2}}J_{\nu}\left(j_{\nu,n}x^\kappa\right),
\end{equation}
where 
\begin{equation}
\nu=\frac{\alpha-1}{2-\alpha},\;\kappa=\frac{2-\alpha}{2}\;\;\text{and}\;J_\nu(x)=\sum_{m\geq 0}\frac{(-1)^m}{m!\Gamma(m+\nu+1)}\left(\frac{x}{2}\right)^{2m+\nu},\;x\geq 0.
\end{equation}
Here $\Gamma(.)$ is the Gamma function, and $\left(j_{\nu,n}\right)_{n\geq 1}$  are the positive zeros of the Bessel function $J_\nu$.\\
Using asymptotic behavior of Bessel functions \cite{ol}, we get
\begin{equation}\label{qss}
\begin{split}
 u_{n}(x)= \frac{\sqrt{2k}}{\vert J_{\nu}^{'}(x) \vert}x^{\frac{1-\alpha}{2}} J_{\nu}(j_{n,\nu}x^{k})&\simeq  \frac{\sqrt{2k}}{\vert J_{\nu}^{'}(x) \vert}x^{\frac{1-\alpha}{2}} \left(\frac{j_{n,\nu}x^{k}}{2}\right)^{\nu} \\
&\simeq \frac{2^{-\nu}\sqrt{k\pi}}{\Gamma(\nu+1)}\left(j_{n,\nu} \right)^{\nu+\frac{1}{2}}. 
\end{split}
\end{equation}
Now, evaluating
\begin{equation}
\begin{split}
\left\|\left(u_n,iu_n\right)\right\|^2_{\mathcal{H}_\alpha}&=\int_0^1x^{\alpha}u_{nx}^2dx+\int_0^1u^2_ndx\\
&=\int_0^1\frac{2^{-2\nu}\sqrt{k\pi}}{\Gamma^2(\nu+1)}\left(j_{n,\nu} \right)^{2\nu+1}dx\\
&=\tilde{M}\lambda_n^{2\nu+1}\\
&\geq \tilde{M}\lambda_n\rightarrow \infty,\quad\text{as}\;n\rightarrow\infty.
\end{split}
\end{equation}
\end{proof}

\section{Numerical simulation of transfer function}
Here we begin by recalling some aspects on input-output systems (see \cite{CZA} for more details). So, let us consider $U,\,X$ be two Hilbert spaces and consider the
abstract control problem
\begin{equation}\label{11}
\left\{
\begin{array}{ll}
\dot{z}(t)=Az(t)+Bu(t),\quad z(0)=z_0\\
y(t)=B^*z(t)\\
\end{array}
\right.
\end{equation}
where $A:D(A)\subset X\longrightarrow X$ generates a
$C_0$-semigroups of contractions $T(t)_{t\geq 0}$,
$B\in\mathcal{L}(U,X)$ is an admissible control operator, $u(.)\in
L^2_{\text{loc}}(0,+\infty; U)$ design the input (or control) function and $y(.)$ design the output (or observation) function. The transfer function of (\ref{11})
is given by $H(\lambda)\in\mathcal{L}(U)$ such that
$$\hat{y}(\lambda)=H(\lambda)\hat{u}(\lambda),$$
where $\hat{.}$ denotes the Laplace transformation. For these
concepts, see \cite{TW}.\\
Now, we consider the following control system 
\begin{equation}\label{e1}
\left\{
\begin{array}{lll}
\begin{split}
\omega_{tt}(x,t)-(x^{\alpha}\omega_{x}(x,t))_{x}=0
\quad & \text{on}\;(0, 1)\times(0, T)\\
(x^{\alpha}\omega_{x})(0,t)=\theta(t),\;\omega(1,t)=0\quad &  t\in (0,T)\\
\omega(x,0)=0,\;\omega_t(x,0)=0\quad & \text{on}\;(0,1)\\
\end{split}
\end{array}\right.
\end{equation}
where $\theta(.)\in L^2(0,T)$. Then, system (\ref{e1}) can be written on the abstract form as (\ref{11}) with $B$ is an unbounded control operator. Hence admissibility of $B$ is not verified and as it was shown in\cite{BLR}, we replace this issue by proving the boundedness of its associated transfer function.
More precisely, we have the following.
\begin{lem}
Let $\gamma>0$ and $C_{\gamma}=\{\lambda\in{\C},\quad \Re e\lambda=\gamma \}$. Then, the transfer function of (\ref{e1}) is given by 
\begin{equation}\label{treee}
\lambda\in C_{\gamma}\rightarrow H(\lambda)=\frac{2((\nu+1)\lambda)^{\nu+1}}{\lambda\Gamma(\nu+1)}\left(((\nu+1)\lambda)^{\nu}\frac{K_{\nu}((\nu+1)\lambda)}{I_{\nu}((\nu+1)\lambda)}-c_{\nu}\right)
\end{equation}
and is bounded on $C_\gamma$, where $I_\nu, K_\nu$ are the modified Bessel functions of first and  second kind and $c_\nu$ is constant complex number.
\end{lem}
\begin{proof}
Applying the Laplace transform to (\ref{e1}) with respect to time $t$ to get $\hat{\omega}(x,\lambda)$ where $\lambda=\gamma+i\kappa$ and $\gamma>0$. Then
\begin{equation}\label{lap}
\left\{
\begin{array}{lll}
 \lambda^{2}\hat{\omega}(x,\lambda)-(x^{\alpha}\hat{\omega}(x,\lambda))_{x}=0,\quad 0<x<1\\
\\
x^{\alpha}\hat{\omega}_{x}(0,\lambda)=\hat{\theta}(\lambda),\quad  \hat{\omega}(1,\lambda)=0.\\
\end{array}
\right. 
\end{equation}
So we obtain the following Sturm-Liouville problem
$$ x^{2}\hat{\omega}_{xx} +\alpha x\hat{\omega}_{x}-\lambda^{2}x^{2-\alpha}\hat{\omega}=0 , $$
with a solution
\begin{equation}\label{besfu}
\hat{\omega}(x,\lambda)=\left\{
\begin{array}{ll}
\begin{split}
x^{\frac{1-\alpha}{2}}\left[A_1I_{\nu}\left(\frac{2\lambda}{2-\alpha}x^{\frac{2-\alpha}{2}}\right)+B_1K_{\nu}\left(\frac{2\lambda}{2-\alpha}x^{\frac{2-\alpha}{2}}\right)\right]\quad&\text{if} \, \nu \in{\N^{*}},\\
x^{\frac{1-\alpha}{2}}\left[A_2I_{\nu}\left(\frac{2\lambda}{2-\alpha}x^{\frac{2-\alpha}{2}}\right)+B_2I_{-\nu}\left(\frac{2\lambda}{2-\alpha}x^{\frac{2-\alpha}{2}}\right)\right]\quad&\text{if} \, \nu \notin{\N^{*}},\\
\end{split}
\end{array}
\right.
\end{equation}
where $A_i,\;B_i$ are complex numbers, $I_\nu,\;K_\nu$ are Bessel functions of first and second kind and $$\nu= \frac{\alpha-1}{2-\alpha}>0.$$
 \underline{First case}: $\nu\in\N^*$.\\
 We shall note that
$$\nu=\frac{\alpha-1}{2-\alpha}\in\mathbb{N}^*\quad\text{if and only if}\quad\alpha\in[\frac{3}{2},2[.$$ For instance, for $0<x\ll\sqrt{\nu+1}$ the following estimation holds
\begin{equation}
\left\{
\begin{array}{ll}
\begin{split}
I_{\nu}(x) &\simeq\frac{1}{\Gamma(\nu+1)}\left(\frac{x}{2}\right)^{\nu}\\
 K_{\nu}(x) &\simeq\frac{\Gamma(\nu)}{2}\left(\frac{2}{x}\right)^{\nu} .\\
\end{split}
\end{array}
\right.
\end{equation}
The second equation in (\ref{lap}) gives us 
\begin{equation}
\left\{
\begin{array}{ll}
\begin{split}
A_1I_{\nu}(\frac{2\lambda}{2-\alpha})+B_1K_{\nu}\left(\frac{2\lambda}{2-\alpha}\right)&=0\\
B_1c_1&=\hat{\theta}(\lambda)\\
\end{split}
\end{array}
\right.
\end{equation}
where 
\begin{equation}\label{cons}
c_{2}=-\frac{\lambda}{2}  \Gamma\left(\frac{1}{2-\alpha}\right)\left(\frac{2-\alpha}{\lambda}\right)^{\frac{1}{2-\alpha}}.
\end{equation}
Thus, constants $A_1$ and $B_1$ are determined by the following expressions 
\begin{equation*}
\left\{
\begin{array}{ll}
\begin{split}
A_{1}&=\frac{(-1)}{c_{2}}\frac{K_{\nu}\left(\frac{2\lambda}{2-\alpha}\right)}{I_{\nu}\left(\frac{2\lambda}{2-\alpha}\right)}\hat{\theta}(\lambda) \\
 B_{1}&=\frac{\hat{\theta}(\lambda)}{c_{2}}\\
\end{split}
\end{array}
\right.
\end{equation*}
where $c_2$ as in (\ref{cons}). Hence,  the  Bessel functions  $K_\nu$ change its shape for each order $\nu$, then using  numerical calculations, we show that the solution $\omega$ defined in (\ref{besfu}) exists and well-defined in a neighborhood of the origin   so that 
\begin{equation*}
 \hat{\omega}(0,\lambda)\simeq\frac{A_{1}}{\Gamma(\nu+1)}(\frac{\lambda}{2-\alpha})^{\nu}+ B_{1}c_{\nu},\quad  \left|c_{\nu}\right|<\infty,
\end{equation*} where $c_{\nu}=\lim\limits_{x \rightarrow 0} x^{\frac{1-\alpha}{2}} K_{\nu}\left(\frac{2\lambda}{2-\alpha}x^{\frac{2-\alpha}{2}}\right)$ is a complex number. By the relation 
\begin{equation}\label{eqz}
\hat{\omega}(0,\lambda)=H(\lambda)\hat{\theta}(\lambda)
\end{equation}
 we can deduce the transfer function which is given by
\begin{equation}
H(\lambda)=\frac{2((\nu+1)\lambda)^{\nu+1}}{ \Gamma(\nu+1)}\left(((\nu+1)\lambda)^{\nu}\frac{K_{\nu}((\nu+1)\lambda)}{I_{\nu}((\nu+1)\lambda)}-c_{\nu}\right) 
\end{equation}

Let $\lambda\in {\C}\backslash\mathbb{R}_{-}$, $\lambda=\gamma+i\kappa ,\ \gamma> 0$.
 The the principal determination of the logarithm of $\lambda$ is defined as follows  
\begin{equation*}
 \log(\lambda)=\ln\vert\lambda\vert + i\arg(\lambda),\quad -\frac{\pi}{2}<\arg(\lambda)\leq\frac{\pi}{2},
\end{equation*}
 
Since 
\begin{equation*}
\begin{split}
\lambda \mapsto \left((\nu+1)\lambda\right)^{\nu+1}&=e^{(\nu+1) Log((\nu+1)\lambda)}\\
&=\omega e^{(\nu+1)Log(\lambda)},\quad\omega=(\nu+1)^{\nu+1},
\end{split}
\end{equation*} we get
\begin{equation*}
 \left|H(\lambda)\right| \leq\frac{2\omega^{2}}{\nu+1}\left|\lambda^{2\nu}\right| \left|\frac{K_{\nu}((\nu+1)\lambda)}{I_{\nu}((\nu+1)\lambda)}\right|+ 2\omega\left|\lambda^{\nu}\right|\left|c_{\nu}\right|. 
\end{equation*}
 where $\mid c_{\nu}\mid $ is finite  and  
\begin{equation*}
 L=\left|\frac{K_{\nu}((\nu+1)\lambda)}{I_{\nu}((\nu+1)\lambda)}\right|\neq0\quad\text{and finite}.
\end{equation*}
 
This shows that $H(\lambda)$ is bounded on ${\C}\backslash\mathbb{R}_{-}$.

\underline{Second case}: $\nu\notin\mathbb{N}^*$.\\
Since $\nu\notin\mathbb{N}^*$, we have $K_{\nu}(x)=I_{-\nu}(x)$. As a result, the solution of (\ref{lap}) becomes
\begin{equation}
\hat{\omega}(x,\lambda)= x^{\frac{1-\alpha}{2}}\left[A_2I_{\nu}\left(\frac{2\lambda}{2-\alpha}x^{\frac{2-\alpha}{2}}\right)+B_2I_{-\nu}\left(\frac{2\lambda}{2-\alpha}x^{\frac{2-\alpha}{2}}\right)\right].
\end{equation}
The boundary conditions  $x^{\alpha}\hat{\omega}_{x}(0,\lambda)=\hat{\theta}(\lambda)$ and $\hat{\omega}(1,\lambda)=0$ allows  us to determine
\begin{equation*}\label{nht}
\left\{
\begin{array}{ll}
\begin{split}
A_{2} &=\frac{I_{\nu}(\frac{2\lambda}{2-\alpha})}{(\alpha-1)I_{\nu}(\frac{2\lambda}{2-\alpha})}\left(\frac{\lambda}{2-\alpha}\right)^{\frac{\alpha-1}{2-\alpha}}\hat{\theta}(\lambda)  \\
B_{2}&=\frac{1}{(1-\alpha)}\left(\frac{\lambda}{2-\alpha}\right)^{\frac{\alpha-1}{2-\alpha}}\hat{\theta}(\lambda).\\
\end{split}
\end{array}
\right.
\end{equation*}
Hence
$$ \hat{\omega}(x,\lambda)\simeq\left(\frac{A_{2}}{\Gamma(\nu+1)}\left((\nu+1)\lambda\right)^{\nu}+ \frac{B_{2}}{\Gamma(1-\nu)}\left((\nu+1)\lambda\right)^{-\nu}x^{\frac{-\nu}{\nu+1}}\right).$$
As $\hat{\omega}(0,\lambda)\rightarrow \infty$, then the transfer function is not bounded.
 
Using relation (\ref{eqz}), we conclude that the transfer function is bounded on $C_{\gamma}$ if $\nu\in\N^*$.
\end{proof}

\begin{figure}
\floatbox[{\capbeside\thisfloatsetup{capbesideposition={left,top},capbesidewidth=9cm}}]{figure}[\FBwidth]
{\caption{value of $\left|c_\nu\right|$ with $\arg(\lambda)=\frac{\pi}{3}$}\label{fig:test}}
{\includegraphics[width=7cm]{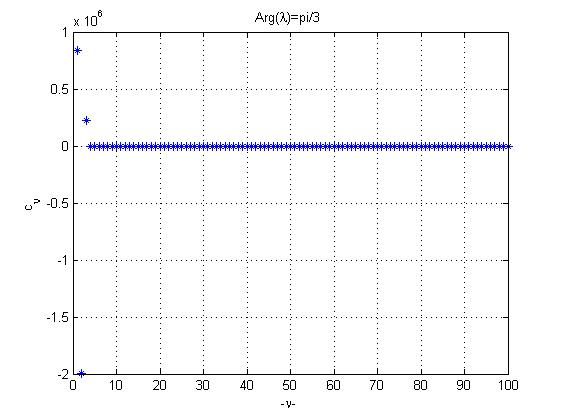}}

\floatbox[{\capbeside\thisfloatsetup{capbesideposition={left,top},capbesidewidth=9cm}}]{figure}[\FBwidth]
{\caption{value of $\left|c_\nu\right|$ with $\arg(\lambda)=\frac{\pi}{4}$}\label{fig:test}}
{\includegraphics[width=7cm]{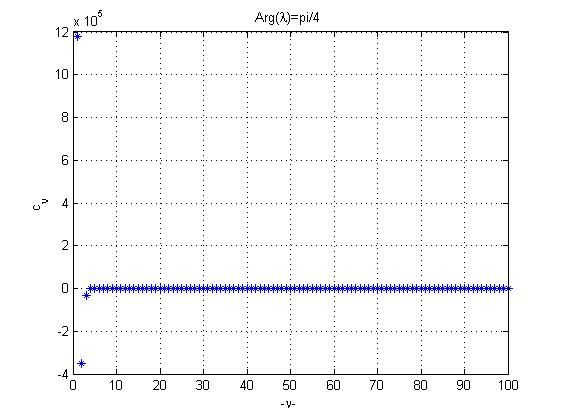}}

\floatbox[{\capbeside\thisfloatsetup{capbesideposition={left,top},capbesidewidth=9cm}}]{figure}[\FBwidth]
{\caption{value of $\left|c_\nu\right|$ with $\arg(\lambda)=\frac{\pi}{6}$}\label{fig:test}}
{\includegraphics[width=7cm]{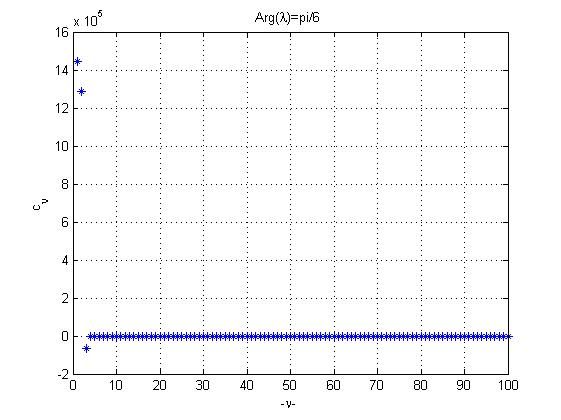}}

\floatbox[{\capbeside\thisfloatsetup{capbesideposition={left,top},capbesidewidth=9cm}}]{figure}[\FBwidth]
{\caption{value of $\left|c_\nu\right|$ with $\arg(\lambda)=\frac{\pi}{2}$}\label{fig:test}}
{\includegraphics[width=7cm]{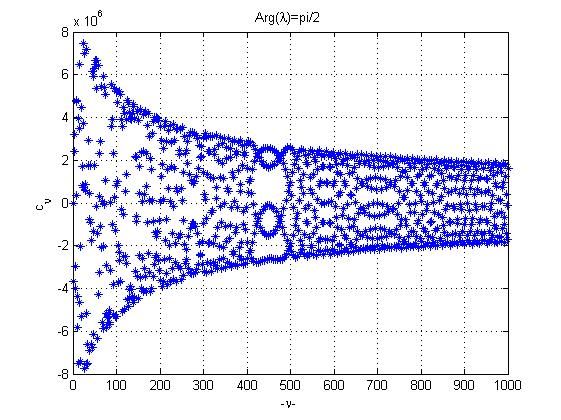}}
\end{figure}

%\begin{figure}
 %\begin{minipage}[b]{.46\linewidth}
  %\centering\epsfig{figure=pi33.jpg,width=1.6\linewidth}
  %\caption{ \label{fig1}}
 %\end{minipage} 
 %\begin{minipage}[b]{.46\linewidth}
  %\centering\epsfig{figure=pi44.jpg,width=1.6\linewidth}
  %\caption{ \label{fig2}}
 %\end{minipage}
%\end{figure}
%\begin{figure}
 %\begin{minipage}[b]{.46\linewidth}
  %\centering\epsfig{figure=pi66.jpg,width=1.6\linewidth}
  %\caption{ \label{fig1}}
 %\end{minipage} \hfill
 %\begin{minipage}[b]{.46\linewidth}
  %\centering\epsfig{figure=pi22.jpg,width=1.6\linewidth}
  %\caption{ \label{fig2}}
 %\end{minipage}
%\end{figure}
%

\end{document}